\documentclass{amsart}

\usepackage{amsmath,amsthm,amssymb,mathtools,amscd,color,comment}
\usepackage[all]{xy}
\usepackage{enumerate}
\usepackage{url}
\usepackage{setspace}
\allowdisplaybreaks

\theoremstyle{definition}
\newtheorem{theorem}{Theorem}[section]
\newtheorem{proposition}[theorem]{Proposition}
\newtheorem{lemma}[theorem]{Lemma}
\newtheorem{corollary}[theorem]{Corollary}

\newtheorem{problem}[theorem]{Problem}

\newtheorem{definition}[theorem]{Definition}

\theoremstyle{remark}
\newtheorem{remark}[theorem]{Remark}
\numberwithin{equation}{section}

\newcommand{\R}{\mathbb{R}}

\newcommand{\N}{\mathbb{N}}

\newcommand{\C}{\mathbb{C}}

\newcommand{\Pol}{\mathrm{Pol}}
\newcommand{\tr}{\mathrm{tr}}

\newcommand{\Span}{\mathrm{Span}}
\newcommand{\Stab}{\mathrm{Stab}}
\newcommand{\supp}{\mathrm{supp}}
\newcommand{\Sym}{\mathrm{Sym}}

\newcommand{\TFF}{\mathrm{TFF}}

\newcommand{\EITFF}{\mathrm{EITFF}}
\newcommand{\ECTFF}{\mathrm{ECTFF}}

\title{Constructing Spherical Designs Using Tight $t$-Fusion Frames}

\author[R.~Misawa]{Ryutaro Misawa}
\address[R.~Misawa]{Graduate School of Information Sciences\\
 Tohoku University\\
6-3-09 Aramaki-Aza-Aoba, Aoba-ku, Sendai 980-8579\\
	 Japan}
\email{misawa.ryutaro.q2@dc.tohoku.ac.jp}
\begin{document}

\begin{abstract}
In this paper, we study conditions under which a finite subset \(Z\) of the unit sphere \(S^{d-1}\subset \R^{d}\) becomes a spherical \(t\)-design, when \(Z\) is constructed by the following procedure: starting from a finite set of \(k\)-dimensional subspaces in the real Grassmannian \(G_{k,d}\), we place, for each such \(k\)-dimensional subspace, a finite set on its unit sphere, and then take the union of these sets in \(S^{d-1}\).
For this construction problem—namely, obtaining spherical designs in higher dimensions by distributing point sets on lower-dimensional spheres subspace by subspace—we provide a sufficient condition based on the framework of tight \(t\)-fusion frames (\(\TFF_t\)) due to Bachoc--Ehler.
As a preparation for applications, we moreover give an explicit construction of equal-weight tight \(2\)-fusion frames on \(G_{2,d}\) for infinitely many dimensions \(d\), via unions of orbits of the hyperoctahedral group.
We also derive necessary conditions for the existence of highly symmetric tight \(t\)-fusion frames, namely equi-chordal and equi-isoclinic tight \(t\)-fusion frames (\(\ECTFF_t\) and \(\EITFF_t\)), on \(G_{2,d}\), and in particular obtain bounds on the number of points.
\end{abstract}

\maketitle

\section{Introduction}

Spherical $t$-designs arise in approximation theory, frame theory, and algebraic combinatorics, among other areas
(\cite{ACSW2010,Womersley2017, BF2003,NPW2006,Renes2004,DGS1977,BB2009}).
They are defined as finite subsets of the $(d-1)$-dimensional unit sphere
\[
S^{d-1}:=\{x\in\R^d\mid \|x\|=1\}
\]
that reproduce spherical averages of all polynomials of degree at most $t$ by finite sums.
This concept was systematically developed by Delsarte--Goethals--Seidel \cite{DGS1977}.
Moreover, Bondarenko--Radchenko--Viazovska \cite{BRV2013} proved that, for each fixed dimension $d$ and every $t\in\N$,
there exists a spherical $t$-design on $S^{d-1}$ with $N=O(t^{d-1})$ points.

Since then, explicit constructions of spherical $t$-designs have been intensively studied
(\cite{Bajnok1991-1, Xiang2022, TTHS2025}).
In this paper, we focus on the following type of construction problem:
\emph{placing designs on lower-dimensional spheres inside each of a family of subspaces, in order to obtain a design on a higher-dimensional sphere}.
Throughout the paper, we write
\[
G_{k,d}:=\{V\le \R^d\mid \dim V=k\}
\]
for the real Grassmann manifold.

\begin{problem}\label{prob:gene}
Fix $t\in\mathbb N$.
For any $(d,k)$ ($d\in\mathbb N$, $1\le k\le d$),
give conditions on a finite set $D\subset G_{k,d}$ and point sets $Y_V\subset S(V)\simeq S^{k-1}$ for each $V\in D$
such that
\[
Z:=\bigcup_{V\in D} Y_V \subset S^{d-1}
\]
is a spherical $t$-design.
In particular, clarify to what extent this property can be controlled by the choice of $D$.
\end{problem}

Constructions of this kind trace back to those of K\"onig \cite{Koenig1999} and Kuperberg \cite{Kuperberg2006}.
Cohn--Conway--Elkies--Kumar \cite{CCEK2007} suggested the potential of constructions along Hopf fibrations.
As a concrete example, Sloane--Hardin--Cara \cite{SHC2003} searched for spherical designs on $S^{3}$
that can be constructed in the case $(d,k)=(4,2)$.
Subsequently, Okuda \cite{Okuda2015} established a theorem for $(d,k)=(4,2)$,
constructing designs on $S^{3}$ via the Hopf map $S^{3}\to S^{2}$.
Furthermore, Lindblad \cite{Lindblad2023} generalized Okuda's framework and treated in a unified manner
the families corresponding to $(d,k)=(2d',2),\ (4d'',4),\ (16,8)$, where $d',d''>1$ are integers.

In this paper, using the framework of tight $t$-fusion frames ($\TFF_t$) introduced by Bachoc--Ehler \cite{BE2013},
we give a general sufficient condition for Problem~\ref{prob:gene}, as well as explicit constructions in a certain class.

In frame theory, a fusion frame is a generalization of an ordinary frame,
obtained by replacing vectors with subspaces.
Concretely, it consists of a family of closed subspaces $\{W_i\}_{i\in I}$ of a Hilbert space $H$
together with weights $w_i>0$.
One calls $(\{W_i\},\{w_i\})$ a fusion frame if the operator $\sum_{i\in I} w_i^{2} P_{W_i}$
(where $P_{W_i}$ denotes the orthogonal projection onto $W_i$)
is bounded, positive, and invertible on $H$ (see \cite{CKL2008}).
Tight $t$-fusion frames ($\TFF_t$) were introduced by Bachoc--Ehler \cite{BE2013} as a higher-order analogue of fusion frames,
and are closely related to designs on Grassmann manifolds.
While the existence of $\TFF_t$ is known in general, explicit and systematic construction methods remain far from complete.
Motivated by this, we pose the following construction problem.

\begin{problem}\label{prob:tff_t}
Let $t$ be a positive integer.
For arbitrary $d,k$, explicitly construct a $\TFF_t$ on $G_{k,d}$.
\end{problem}

Finally, inspired by the strong symmetry exhibited by equi-angular line sets on $G_{1,d}$,
we consider higher-dimensional analogues on $G_{k,d}$:
equi-chordal tight $t$-fusion frames ($\ECTFF_t$) and equi-isoclinic tight $t$-fusion frames ($\EITFF_t$).
These are highly symmetric subclasses of $\TFF_t$ in which the inter-subspace distances are uniform,
and they are closely connected to optimal packings and highly symmetric configurations in Grassmannian spaces.
In this paper, we derive necessary parameter conditions 
for the existence of $\ECTFF_2$ and $\EITFF_2$.

\subsection*{Contributions}
\begin{enumerate}
  \item For Problem~\ref{prob:gene}, we provide a general sufficient condition using $\TFF_t$ on $G_{k,d}$.
  \item For Problem~\ref{prob:tff_t}, we give an explicit construction method for $\TFF_2$ on $G_{2,d}$.
  \item We establish bounds on the cardinalities of $\ECTFF_2$/$\EITFF_2$ on $G_{2,d}$.
\end{enumerate}

The organization of this paper is as follows.
In Section~2, we introduce $\TFF_t$ and related notions, and explain their relation to designs on Grassmann manifolds.
In Section~3, we prove a theorem that produces spherical designs from $\TFF_t$ on $G_{k,d}$.
In Section~4, we explicitly construct $\TFF_2$ on $G_{2,d}$ by taking unions of orbits of the hyperoctahedral group.
In Section~5, we discuss non-existence results and necessary conditions for $\ECTFF_2$ and $\EITFF_2$.


\section{Preliminaries}\label{sec:prelim}

\subsection{Spherical designs}\label{subsec:sph-design}

Let $d\in\mathbb N$.
We write $d\sigma$ for the normalized surface-area measure on $S^{d-1}$.
We also introduce the Pochhammer symbol
\[
(a)_0:=1,\qquad (a)_m:=a(a+1)\cdots(a+m-1)\quad(m\ge1).
\]

\begin{definition}\cite[Def.~5.1]{DGS1977}\label{def:spherical-design-cubature}
Let $t$, $d\in\mathbb N$. A finite set $X=\{x_1,\dots,x_n\}\subset S^{d-1}$ is called a (\emph{equal-weight}) \emph{spherical $t$-design} if, for every function 
$f:\mathbb R^d\to\mathbb R$ of total degree at most $t$, one has
\[
\int_{S^{d-1}} f(x)\,d\sigma(x)=\frac1n\sum_{j=1}^n f(x_j).
\]
\end{definition}

To state an equivalent characterization, we introduce scaled Gegenbauer polynomials.
Assume $d\ge 2$. For $\ell\in\mathbb Z_{\ge0}$, define a sequence of polynomials $Q_\ell^{(d)}(x)$ by
\[
Q_0^{(d)}(x):=1,\qquad Q_1^{(d)}(x):=x,
\]
and by the recurrence relation
\[
(\ell+d-2)\,Q_{\ell+1}^{(d)}(x)
=(2\ell+d-2)\,x\,Q_\ell^{(d)}(x)-\ell\,Q_{\ell-1}^{(d)}(x)
\qquad(\ell\ge1).
\]
Then $Q_\ell^{(d)}(1)=1$ holds for all $\ell$.

\begin{proposition}\cite{BB2009}\label{prop:sph-design-equiv}
Let $t$, $d\in\mathbb N$ and $X=\{x_1,\dots,x_n\}\subset S^{d-1}$. The following are equivalent:
\begin{enumerate}
\item[\rm (i)]
$X$ is a spherical $t$-design.
\item[\rm (ii)]
For every $1\le \ell\le t$,
\[
\sum_{i,j=1}^n Q_\ell^{(d)}\!\bigl(\langle x_i,x_j\rangle\bigr)=0
\]
holds, where $\langle\cdot,\cdot\rangle$ is the standard inner product on $\R^d$.
\item[\rm (iii)]
For every $y\in\mathbb R^d$ and every integer $0\le m\le t$,
\[
\sum_{j=1}^n \langle x_j,y\rangle^{m}
=
n\int_{S^{d-1}}\langle x,y\rangle^{m}\,d\sigma(x)
=
\begin{cases}
\displaystyle
n\,\frac{\left(\frac12\right)_{m/2}}{\left(\frac d2\right)_{m/2}}\,
\|y\|^{m}
& \text{if $m$ is even},\\[2ex]
0 & \text{if $m$ is odd}
\end{cases}
\]
holds.
\end{enumerate}
\end{proposition}

As a natural generalization of spherical designs, we introduce the weighted version.

\begin{definition}\label{def:weighted-sph-design}
Let $t$, $d\in\N$.
A pair consisting of a finite set $X=\{x_1,\dots,x_n\}\subset S^{d-1}$ and positive weights
$\lambda_1,\dots,\lambda_n>0$ is called a \emph{weighted spherical $t$-design} if, for every function
$f:\R^d\to\R$ of total degree at most $t$, one has
\[
\int_{S^{d-1}} f(x)\,d\sigma(x)
=
\frac{1}{\Lambda}\sum_{j=1}^n \lambda_j\, f(x_j),
\qquad
\Lambda:=\sum_{j=1}^n \lambda_j.
\]
\end{definition}

\begin{remark}\label{rem:weighted-to-equalweight}
If the weights are constant, then a weighted spherical $t$-design is the same as an (equal-weight) spherical $t$-design.
\end{remark}

\begin{lemma}\label{lem:weighted-sph-moment}
Let $t$, $d\in\N$, and let $X=\{x_1,\dots,x_n\}\subset S^{d-1}$ be a finite set with positive weights
$\lambda_1,\dots,\lambda_n>0$. Set $\Lambda:=\sum_{j=1}^n\lambda_j$.
Then the following are equivalent:
\begin{enumerate}
\item[\rm (i)]
$(X,\{\lambda_j\})$ is a weighted spherical $t$-design.
\item[\rm (ii)]
For every $y\in\R^d$ and every integer $0\le m\le t$,
\[
\sum_{j=1}^n \lambda_j\,\langle x_j,y\rangle^{m}
=
\Lambda\int_{S^{d-1}}\langle x,y\rangle^{m}\,d\sigma(x)
=
\begin{cases}
\displaystyle
\Lambda\,
\frac{\left(\frac12\right)_{m/2}}{\left(\frac d2\right)_{m/2}}\,
\|y\|^{m}
& \text{if $m$ is even},\\[2ex]
0 & \text{if $m$ is odd}.
\end{cases}
\]
\end{enumerate}
\end{lemma}

The proof is the same as the standard argument in the equal-weight case, and we omit it since it only requires replacing the average by the weighted average.

\subsection{Grassmann manifolds, principal angles, and equiangularity}\label{subsec:grass-geometry}

We will work with finite subsets of Grassmann manifolds.
As an analogue of the inner product for point configurations on spheres, we recall the notion of principal angles. Henceforth, we fix integers $t,d,k\in\mathbb N$ with $1\le k\le d$.


Let $V,W\in G_{k,d}$. Choose orthonormal bases $v_1,\dots,v_k$ of $V$ and $w_1,\dots,w_k$ of $W$, and set
\[
Q_V:=\bigl[v_1\ \cdots\ v_k\bigr],\qquad
Q_W:=\bigl[w_1\ \cdots\ w_k\bigr]\in\R^{d\times k},
\]
so that $Q_V^{\mathsf T}Q_V=Q_W^{\mathsf T}Q_W=I_k$.

\begin{proposition}\label{prop:sv-bound}
Let $A:=Q_V^{\mathsf T}Q_W$. Its singular values $\sigma_1\ge\cdots\ge\sigma_{k}\ge0$ satisfy $0\le\sigma_i\le 1$.
\end{proposition}

\begin{proof}
For any $x\in\R^k$,
\[
\|Ax\|=\|Q_V^{\mathsf T}(Q_Wx)\|\le \|Q_Wx\|=\|x\|,
\]
since $Q_V^{\mathsf T}$ is a contraction (its rows are orthonormal) and $Q_W$ is an isometry.
Hence the largest singular value satisfies $\sigma_1=\max_{\|x\|=1}\|Ax\|\le 1$, so $0\le\sigma_i\le\sigma_1\le 1$.
\end{proof}

\begin{definition}\cite[Section~2]{BCN2002}\label{def:principal-angles}
Let $\sigma_1\ge\cdots\ge\sigma_{k}\ge0$ be the singular values of $A=Q_V^{\mathsf T}Q_W$.
By Proposition~\ref{prop:sv-bound}, $\sigma_i\in[0,1]$, and we define
\[
\theta_i(V,W):=\arccos(\sigma_i)\in\Bigl[0,\frac{\pi}{2}\Bigr]\qquad(i=1,\dots,k).
\]
These are called the \emph{principal angles} between $V$ and $W$.
\end{definition}

\begin{remark}\label{rem:principal-angles-well-defined}
If we choose different orthonormal bases, then $Q_V'=Q_VU$ and $Q_W'=Q_WV$ for some $U,V\in O(k)$, and hence
\[
(Q_V')^{\mathsf T}Q_W' = U^{\mathsf T}(Q_V^{\mathsf T}Q_W)V.
\]
Since singular values are invariant under left and right multiplication by orthogonal matrices, the definition of
principal angles is independent of the choice of orthonormal bases.
\end{remark}

\begin{definition}\label{def:EI}
A finite set $D=\{V_i\}_{i=1}^N\subset G_{k,d}$ is called \emph{equi-isoclinic} (EI) if there exists $\theta\in[0,\pi/2]$
such that for any distinct $i\neq j\in[N]$ and any $\ell\in[k]$,
\[
\theta_\ell(V_i,V_j)=\theta
\]
holds.
\end{definition}

\begin{definition}\cite{CHS1996}\label{def:chordal}
For $V,W \in G_{k,d}$,
\[
d_C(V,W):=\sqrt{\sum_{\ell=1}^k \sin^2\!\theta_\ell(V,W)}.
\]
This quantity is called the \emph{chordal distance} between $V$ and $W$.
\end{definition}

\begin{definition}\label{def:EC}
A finite set $D=\{V_i\}_{i=1}^N\subset G_{k,d}$ is called \emph{equi-chordal} (EC) if for any distinct $i\neq j\in[N]$,
\[
d_C(V_i,V_j)=\text{\rm const.}
\]
\end{definition}

\subsection{Grassmannian designs}\label{subsec:grass-design}

We write $d\sigma_{k,d}$ for the $O(d)$-invariant probability measure on the Grassmannian $G_{k,d}$.

\begin{proposition}\cite[Thm.~3.2, Def.~3.3]{BCN2002}\label{prop:Irr}
The space $L^2(G_{k,d})$ decomposes into irreducible $O(d)$-subrepresentations as
\[
L^2(G_{k,d})
\;=\;
\bigoplus_{\ell(\mu)\le k} \,  \mathsf{H}^{\, k, d}_{2\mu}.
\]
Here $\mu=(\mu_1,\dots,\mu_r)$ is a partition (i.e., $\mu_1\ge\cdots\ge\mu_r\ge1$) or the empty partition $\mu=\varnothing$.
We set $\deg(\mu):=\sum_{i=1}^r\mu_i$ and $\ell(\mu):=r$, with the convention $\deg(\varnothing)=\ell(\varnothing)=0$.
Moreover, $\mathsf{H}^{\,k,d}_{2\mu}\subset L^2(G_{k,d})$ denotes the irreducible component of type
$2\mu:=(2\mu_1,\dots,2\mu_r)$.
\end{proposition}

For $t\in\N$, define a subspace of functions on $G_{k,d}$ by
\[
\Pol_{2t}(G_{k,d})=\Pol_{2t+1}(G_{k,d})
\;:=\;
\bigoplus_{\substack{\ell(\mu)\le k\\ \deg(\mu)\le t}} \mathsf{H}^{\,k,d}_{2\mu}.
\]
We further set
\[
\Pol^{(1)}_{2t}(G_{k,d})=\Pol^{(1)}_{2t+1}(G_{k,d})
\;:=\;
\bigoplus_{1\le j\le t} \mathsf{H}^{\,k,d}_{(2j)}
\;\subset\; \Pol_{2t}(G_{k,d}).
\]


\begin{definition}\cite[Prop.~4.2]{BCN2002}\label{def:Grass-design}
A finite set $D=\{V_j\}_{j=1}^n\subset G_{k,d}$ is called a \emph{Grassmann $2t$-design} if, for every function
$f\in \Pol_{2t}(G_{k,d})$, one has
\[
\int_{G_{k,d}} f(V)\, d\sigma_{k,d}(V)=\frac1n\sum_{j=1}^n f(V_j).
\]
\end{definition}

There are several equivalent characterizations of Grassmann $2t$-designs. In particular, we will use the criterion via zonal polynomials.
For $V,W\in G_{k,d}$, set
\[
y_i(V,W):=\cos^2\theta_i(V,W)\qquad(1\le i\le k).
\]
Fix a base point $V_0\in G_{k,d}$. Its stabilizer satisfies
\(
\Stab(V_0)\simeq O(k)\times O(d-k).
\)
Since each irreducible component $\mathsf H^{k,d}_{2\mu}$ ($\ell(\mu)\le k$) occurs in $L^2(G_{k,d})$ with multiplicity one,
the $\,\Stab(V_0)$-fixed subspace
\[
\bigl(\mathsf H^{k,d}_{2\mu}\bigr)^{\Stab(V_0)}
:=\{\,f\in \mathsf H^{k,d}_{2\mu}\mid f(g\cdot V)=f(V)\ \ (\forall g\in\Stab(V_0),\ \forall V\in G_{k,d})\,\}
\]
is one-dimensional. Therefore there is a nonzero $\Stab(V_0)$-invariant function $f$ in $\mathsf H^{k,d}_{2\mu}$, unique up to scalar multiple.
We call such a function $f$ a \emph{zonal polynomial} of type $2\mu$ with respect to $V_0$. We then define $P_{2\mu}=f/f(V_0)\in \bigl(\mathsf H^{k,d}_{2\mu}\bigr)^{\Stab(V_0)}.$ By $\Stab(V_0)$-invariance, $P_{2\mu}(V)$ depends only on $y_i(V,V_0)=\cos^2\theta_i(V,V_0)$, and hence it may be regarded as a symmetric polynomial in $k$ variables, which we also denote by $P_{2\mu}$, so that
\[
P_{2\mu}(V)=P_{2\mu}\bigl(y_1(V,V_0),\dots,y_k(V,V_0)\bigr).
\]
In particular, the normalization $P_{2\mu}(V_0)=1$ is equivalent to $P_{2\mu}(1,\dots,1)=1$.
Moreover, for arbitrary $V,W\in G_{k,d}$ we set
\[
P_{2\mu}(V,W):=P_{2\mu}\bigl(y_1(V,W),\dots,y_k(V,W)\bigr).
\]
\begin{remark}\label{rem:grass-zonal-low}
Assume $k\ge2$; the component of type $(2,2)$ appears only in this range (for $k=1$ the term $P_{(2,2)}$ is absent).
We list the explicit low-degree zonal polynomials $P_{(2)},P_{(4)},P_{(2,2)}$.
(see \cite{JC1974}).

\begin{equation}\label{eq:P_2}
P_{(2)}=\frac{k}{k-d}-\frac{d}{k-d}\,c_2,
\qquad \text{where}\qquad
c_2=\frac1k\sum_{i=1}^k y_i.
\end{equation}

\begin{equation}\label{eq:P_4}
P_{(4)}=\frac{P'_{(4)}}{P'_{(4)}(1,\dots,1)},
\qquad \text{where}\qquad
\begin{aligned}
P'_{(4)}
&=
1-\frac{2(d+2)}{k}\,c_2+\frac{(d+2)(d+4)}{k(k+2)}\,c_4,\\
c_4
&=\frac{3}{k(k+2)}
\left(
\sum_{i=1}^k y_i^2+\frac{2}{3}\sum_{1\le i<j\le k}y_iy_j
\right).
\end{aligned}
\end{equation}

\begin{equation}\label{eq:P_22}
P_{(2,2)}=\frac{P'_{(2,2)}}{P'_{(2,2)}(1,\dots,1)},
\qquad \text{where}\qquad
\begin{aligned}
P'_{(2,2)}
&=
1-\frac{2(d-1)}{k}\,c_2+\frac{(d-1)(d-2)}{k(k-1)}\,c_{22},\\
c_{22}
&=\frac{2}{k(k-1)}\sum_{1\le i<j\le k}y_iy_j.
\end{aligned}
\end{equation}

\end{remark}

\begin{proposition}\cite[Prop.~4.2]{BCN2002}\label{prop:grass-design-equiv}
Let $D\subset G_{k,d}$. The following are equivalent:
\begin{enumerate}
\item[\rm (i)]
$D$ is a Grassmann $2t$-design.
\item[\rm (ii)]
For every nonzero partition $\mu$ satisfying $\deg(\mu)\le t$,
\[
\sum_{V,W \in D} P_{2\mu}\bigl(V,W\bigr)=0
\]
holds.
\end{enumerate}
\end{proposition}

\subsection{Tight $t$-fusion frames}\label{subsec:fusionframe}

\begin{definition}\cite[Def.~4.1]{BE2013}\label{def:p-fusion-frame}
Let $\{V_j\}_{j=1}^n\subset G_{k,d}$ be a finite set and let $\{\omega_j\}_{j=1}^n$ be positive weights.
For each $j$, let $P_{V_j}:\mathbb{R}^d\to\mathbb{R}^d$ denote the orthogonal projection onto $V_j$.
If there exist constants $A,B>0$ such that for all $x\in\mathbb{R}^d$,
\[
A\,\|x\|^{2t}\ \le\ \sum_{j=1}^n \omega_j\,\|P_{V_j}x\|^{2t}\ \le\ B\,\|x\|^{2t},
\]
then $\{(V_j,\omega_j)\}_{j=1}^n$ is called a \emph{$t$-fusion frame}.
If, in particular, $A=B$, then it is called a \emph{tight $t$-fusion frame}.
In the equal-weight case, we call $\{(V_j,\omega_j)\}_{j=1}^n$ an \emph{equal-weight tight $t$-fusion frame} and often omit the weights, simply writing $\{V_j\}_{j=1}^n$.

\end{definition}

Bachoc--Ehler \cite{BE2013} characterized tight $t$-fusion frames as cubature formulas for $\Pol^{(1)}_{2t}(G_{k,d})$.

\begin{theorem}\cite[Thm.~5.3]{BE2013}\label{thm:BE-cubature}
Let $D\subset G_{k,d}$ and positive weights $\{\omega_V\}_{V\in D}$, and set
\[
\Omega:=\sum_{V\in D}\omega_V.
\]
Then the following are equivalent:
\begin{enumerate}
\item[\rm (i)]
$\{(V,\omega_V)\}_{V \in D}$ is a tight $t$-fusion frame.
\item[\rm (ii)]
For every $f\in \Pol^{(1)}_{2t}(G_{k,d})$,
\begin{equation}\label{eq:BE-cubature}
\sum_{V \in D} \omega_V\, f(V)
=
\Omega\int_{G_{k,d}} f(V)\, d\sigma_{k,d}(V).
\end{equation}
Equivalently, with $\tilde\omega_V:=\omega_V/\Omega$ one has
\[
\int_{G_{k,d}} f(V)\, d\sigma_{k,d}(V)=\sum_{V\in D}\tilde\omega_V f(V).
\]
\item[\rm (iii)]
For each $1\le \ell\le t$,
\[
\sum_{V,W \in D} \omega_V\omega_W\,
P_{(2\ell)}\!\bigl(V,W\bigr)=0
\]
holds.
\end{enumerate}
\end{theorem}

\begin{remark}\label{rem:design-implies-tff}
By Definition~\ref{def:Grass-design} and Theorem~\ref{thm:BE-cubature}, for any $t$ we have
\[
\text{Grassmann }2t\text{-design}\ \Longrightarrow\ \text{tight }t\text{-fusion frame},
\]
since $\Pol^{(1)}_{2t}(G_{k,d})\subset \Pol_{2t}(G_{k,d})$.
In general, the converse does not hold.
In particular, for $k\ge2$ and $t\ge2$ one has
\(
\Pol^{(1)}_{2t}(G_{k,d})\subsetneq \Pol_{2t}(G_{k,d})
\),
so tight $t$-fusion frames form a strictly weaker notion than Grassmann $2t$-designs.
(When $k=1$ or $t=1$, we have $\Pol^{(1)}_{2t}(G_{k,d})=\Pol_{2t}(G_{k,d})$.)
\end{remark}

\section{Spherical Designs Using Tight $t$-Fusion Frames}

In this section we present a general principle for constructing (weighted) spherical designs on higher-dimensional spheres
from tight $t$-fusion frames on $G_{k,d}$.

For each $V\in G_{k,d}$, write
\[
S(V):=\{x\in V:\|x\|=1\}\simeq S^{k-1}.
\]
Let $d\sigma_V$ denote the $O(V)\simeq O(k)$-invariant measure on $S(V)$.

For $y\in S^{d-1}$ and $m\in\N$, define
\[
f_{y,m}(V):=\bigl(\tr(P_VP_y)\bigr)^m\qquad (V\in G_{k,d}),
\]
where $P_V$ is the orthogonal projection onto $V$, and $P_y$ is the orthogonal projection onto the line $\R y$.

\begin{lemma}[{\cite[Remark~5.2, Thm.~5.3]{BE2013}}]\label{lem:trace_moment_and_pol1_en}
Let $d\in\N$ and $1\le k\le d$. For any $y\in S^{d-1}$ and $m\in\N$, the following hold:
\begin{enumerate}
\item $f_{y,m}\in \Pol^{(1)}_{2m}(G_{k,d})$. In particular, if $m\le t$, then $f_{y,m}\in \Pol^{(1)}_{2t}(G_{k,d})$.
\item
\[
  \int_{G_{k,d}} \bigl(\tr(P_VP_y)\bigr)^m\, d\sigma_{k,d}(V)
  =
  \frac{\bigl(\frac{k}{2}\bigr)_m}{\bigl(\frac{d}{2}\bigr)_m}.
\]
\end{enumerate}
\end{lemma}

\begin{theorem}\label{thm:lifting-weighted}
Let $t,s\in\N$. Let $D\subset G_{k,d}$ be a finite set with positive weights $\{\omega_V\}_{V\in D}$, and assume that
$(D,\{\omega_V\})$ is a tight $t$-fusion frame. Set
\[
\Omega:=\sum_{V\in D}\omega_V.
\]
For each $V\in D$, let $(Y_V,\{\lambda_{V,z}\}_{z\in Y_V})$ be a weighted spherical $s$-design on $S(V)$, and assume that
\[
\Lambda:=\sum_{z\in Y_V}\lambda_{V,z}
\]
is independent of $V$.
Define a weighted finite set $(Z,w)$ on $S^{d-1}$ as follows:
\[
Z:=\bigcup_{V\in D} Y_V \subset S^{d-1},
\qquad
w(z):=\sum_{\substack{V\in D\\ z\in Y_V}}\omega_V\,\lambda_{V,z}\quad (z\in Z).
\]
Then $(Z,w)$ is a weighted spherical $r$-design on $S^{d-1}$ with
\(
r=\min\{s,\,2t+1\}.
\)
\end{theorem}

\begin{proof}
By Lemma~\ref{lem:weighted-sph-moment}, it suffices to show that for every $y\in\R^d$ and every integer $0\le \ell\le r$,
\begin{equation}\label{eq:moment-identities}
\sum_{z\in Z} w_z\,\langle z,y\rangle^{\ell}
=
\begin{cases}
\Gamma\,\|y\|^{2m}\,\dfrac{(1/2)_m}{(d/2)_m} &\text{if } \ell=2m \text{ is even},\\[1.2ex]
0 & \text{if } \ell \text{ is odd}.
\end{cases}
\end{equation}
where
\[
\Gamma:=\sum_{z\in Z}w_z
=\sum_{V\in D}\omega_V\sum_{z\in Y_V}\lambda_{V,z}
=\Omega\Lambda.
\]

Fix $y\in\R^d$ and $0\le \ell\le r$.
By definition,
\[
\sum_{z\in Z} w_z\,\langle z,y\rangle^{\ell}
=\sum_{V\in D}\omega_V\sum_{z\in Y_V}\lambda_{V,z}\,\langle z,y\rangle^{\ell}.
\]
Since $z\in Y_V\subset V$, we have $P_Vz=z$, and hence
\begin{equation}\label{eq:basic-rewrite-fixed-en}
\sum_{z\in Z} w_z\,\langle z,y\rangle^{\ell}
=\sum_{V\in D}\omega_V\sum_{z\in Y_V}\lambda_{V,z}\,\langle z,P_Vy\rangle^{\ell}.
\end{equation}

\noindent\textbf{(A) The case where $\ell$ is odd.}
Since $\ell\le r\le s$, for each $V$ the function $x\mapsto \langle x,P_Vy\rangle^\ell$ on $S(V)$ is odd, and therefore
\[
\int_{S(V)}\langle x,P_Vy\rangle^\ell\,d\sigma_V(x)=0.
\]
Moreover, since $(Y_V,\{\lambda_{V,z}\})$ is a weighted spherical $s$-design and $\ell\le s$, we have
\[
\sum_{z\in Y_V}\lambda_{V,z}\,\langle z,P_Vy\rangle^\ell
=\Lambda\int_{S(V)}\langle x,P_Vy\rangle^\ell\,d\sigma_V(x)=0.
\]
Substituting this into \eqref{eq:basic-rewrite-fixed-en} yields
$\sum_{z\in Z} w_z\,\langle z,y\rangle^\ell=0$.

\noindent\textbf{(B) The case where $\ell=2m$ is even.}
Since $2m\le r\le s$, we have $2m\le s$, and since $2m\le r\le 2t+1$, we have $m\le t$.
For each $V\in D$, since $(Y_V,\{\lambda_{V,z}\})$ is a weighted spherical $s$-design on $S(V)\simeq S^{k-1}$
and $2m\le s$, applying Lemma~\ref{lem:weighted-sph-moment} on $S(V)$ (identified with $S^{k-1}$) with $d=k$ and $y=P_Vy$ gives

\[
\sum_{z\in Y_V}\lambda_{V,z}\,\langle z,P_Vy\rangle^{2m}
=\Lambda\,\|P_Vy\|^{2m}\,\frac{(1/2)_m}{(k/2)_m}.
\]
Therefore,
\[
\sum_{z\in Z} w_z\,\langle z,y\rangle^{2m}
=\Lambda\,\frac{(1/2)_m}{(k/2)_m}\sum_{V\in D}\omega_V\,\|P_Vy\|^{2m}.
\]
Let $P_y$ denote the orthogonal projection onto the line $\R y$. Since $\|P_Vy\|^2=\|y\|^2\,\tr(P_VP_y)$, we obtain
\[
\|P_Vy\|^{2m}=\|y\|^{2m}\bigl(\tr(P_VP_y)\bigr)^m.
\]
Hence
\[
\sum_{z\in Z} w_z\,\langle z,y\rangle^{2m}
=\Lambda\,\frac{(1/2)_m}{(k/2)_m}\,\|y\|^{2m}
\sum_{V\in D}\omega_V\,\bigl(\tr(P_VP_y)\bigr)^m.
\]
By Lemma~\ref{lem:trace_moment_and_pol1_en}, the function $V\mapsto(\tr(P_VP_y))^m$ lies in $\Pol^{(1)}_{2t}(G_{k,d})$,
so \eqref{eq:BE-cubature} yields

\[
\sum_{V\in D}\omega_V\,\bigl(\tr(P_VP_y)\bigr)^m
=\Omega\int_{G_{k,d}}\bigl(\tr(P_VP_y)\bigr)^m\,d\sigma_{k,d}(V)
=\Omega\,\frac{(k/2)_m}{(d/2)_m}.
\]
Substituting this into the previous identity gives
\[
\sum_{z\in Z} w_z\,\langle z,y\rangle^{2m}
=\Omega\Lambda\,\|y\|^{2m}\,\frac{(1/2)_m}{(d/2)_m}
=\Gamma\,\|y\|^{2m}\,\frac{(1/2)_m}{(d/2)_m}.
\]
Combining {\rm(A)} and {\rm(B)}, \eqref{eq:moment-identities} holds for all $0\le \ell\le r$.
Therefore $(Z,w)$ is a weighted spherical $r$-design.
\end{proof}

\begin{corollary}\label{cor:lifting}
Let $t,s\in\N$.
Assume that $D\subset G_{k,d}$ is an equal-weight tight $t$-fusion frame, and that for each $V\in D$,
$Y_V\subset S(V)$ is a spherical $s$-design, with $|Y_V|$ independent of $V$. 
Let
\[
Z:=\bigcup_{V\in D} Y_V\subset S^{d-1}
\]
be viewed as a multiset.
Then the multiset $Z$ is a spherical $\min\{s,\,2t+1\}$-design on $S^{d-1}$.
\end{corollary}
\begin{remark}\label{rem:disjointness-short-en}
Corollary~\ref{cor:lifting} gives a spherical $r$-design in the \emph{multiset} sense.
If, in addition, $Y_V\cap Y_{V'}=\varnothing$ for $V\neq V'$, then
$Z=\bigsqcup_{V\in D}Y_V$ and we obtain an ordinary spherical design. If $k\ge2$, we may retake $Y_V$'s to make this additional condition satisfied, as follows. Fix $V\in D$ and suppose $Y_V \cap Y_{V'} \neq \varnothing$ for some $V' \in D$. 
Define 
\[
\mathcal B
=
\bigcup_{V'\in D\setminus\{V\}}\bigl\{\, g\in O(V)\ \bigm|\ (gY_V)\cap Y_{V'}\neq\varnothing\,\bigr\}.
\]
Then $\mathcal B$ is a finite union of cosets of subgroups conjugate to $O(k-1)$ in $O(V)$.
Each such coset is a closed subset of $O(V)$ with empty interior since $k\ge2$, hence its complement is open dense.
Therefore $O(V)\setminus\mathcal B$ is a finite intersection of open dense subsets of $O(V)$, and in particular is non-empty.
Then choosing $g\in O(V)\setminus\mathcal B$ yields
$(gY_V)\cap Y_{V'}=\varnothing$ for all $V'\neq V$. Since $gY_V\subset S(V)$ is again a spherical $s$-design, we may replace $Y_V$ with $g Y_V$. Repeating this procedure for each $V \in D$, we obtain a pairwise disjoint family $\{Y_V\}_{V \in D}$.
\end{remark}

To conclude this section, we note that the construction of spherical $(2t+1)$-designs using Corollary~\ref{cor:lifting} consists of two independent tasks:
\begin{enumerate}
\item constructing spherical $(2t+1)$-designs on the lower-dimensional sphere $S^{k-1}$, and
\item constructing tight $t$-fusion frames on $G_{k,d}$.
\end{enumerate}
\section{Constructing Tight $2$-Fusion Frames on $G_{2,d}$ via Group Orbits}
\label{sec:Bd-orbit-G2d}

In this section we construct equal-weight tight $2$-fusion frames on $G_{2,d}$ from orbits of the hyperoctahedral group
$B_d$. These Grassmannian configurations can then be used as inputs for the lifting result
(Corollary~\ref{cor:lifting}) to produce spherical designs.
Throughout, binomial coefficients are interpreted by
\[
\binom{n}{k}=0\qquad (k<0\ \text{or}\ k>n).
\]Let $d\ge 4$, and let $e_1,\dots,e_d$ denote the standard basis of $\R^d$.
The hyperoctahedral group
\[
B_d:=\{\pm1\}^d\rtimes S_d \subset O(d)
\]
acts on $\R^d$ by signed coordinate permutations, and hence acts on $G_{2,d}$ by $g\cdot W:=gW$.
For a finite set $X\subset G_{2,d}$, define
\[
F_X(x):=\frac1{|X|}\sum_{W\in X}\|P_Wx\|^4 \qquad (x\in\R^d),
\]
where $P_W$ denotes the orthogonal projection onto $W$.

\begin{proposition}
\label{prop:tff2-iff-const}
For a finite set $X\subset G_{2,d}$, the following are equivalent:
\begin{enumerate}
\item[\rm (i)]
$X$ is a $\TFF_2$.
\item[\rm (ii)]
$F_X$ is constant on the unit sphere $S^{d-1}$.
\end{enumerate}
\end{proposition}

\begin{proof}
The implication (i)$\Rightarrow$(ii) is immediate.
For (ii)$\Rightarrow$(i), let $x\neq 0$ and set $u:=x/\|x\|\in S^{d-1}$.
By $4$-homogeneity of $F_X$,
\[
F_X(x)=F_X(\|x\|u)=\|x\|^4F_X(u).
\]
By (ii), $F_X(u)$ is a constant $C$ on the sphere, hence $F_X(x)=C\|x\|^4$ for all $x\neq 0$.
The case $x=0$ is trivial.
\end{proof}

In what follows, we consider criteria for a $B_d$-invariant subset $X\subset G_{2,d}$ to be a $\TFF_2$.

\begin{lemma}
\label{lem:bd-inv-form}
If $X\subset G_{2,d}$ is $B_d$-invariant, then $F_X$ is $B_d$-invariant, i.e.,
$F_X(gx)=F_X(x)$ for all $g\in B_d$ and $x\in \R^d$.
Moreover, there exist constants $\alpha,\beta\in\R$ such that
\[
F_X(x)=\alpha\sum_{i=1}^d x_i^4+\beta\Bigl(\sum_{i=1}^d x_i^2\Bigr)^2
\qquad(\forall x\in\R^d).
\]
\end{lemma}

\begin{proof}
From $gX=X$ and the orthogonality identity $P_{gW}=gP_Wg^{\mathsf T}$
(hence $\|P_{gW}x\|=\|P_W(g^{\mathsf T}x)\|$), we obtain
\[
F_X(g^{\mathsf T}x)=F_X(x)\qquad(\forall g\in B_d),
\]
so in particular $F_X$ is invariant under coordinate permutations and sign changes.

Sign-invariance implies that $F_X$ can be written as a polynomial in $y_i:=x_i^2$.
Since $F_X$ is homogeneous of degree $4$ in $x$, it is homogeneous of degree $2$ in $y$.
Moreover, permutation-invariance implies that $F_X$ is a symmetric homogeneous polynomial of degree $2$
in $y_1,\dots,y_d$.
The space of symmetric homogeneous polynomials of degree $2$ is spanned by
\[
\sum_{i=1}^d y_i^2,\qquad \Bigl(\sum_{i=1}^d y_i\Bigr)^2,
\]
which yields the claimed form.
\end{proof}
By Lemma~\ref{lem:bd-inv-form}, on the sphere we may write
$F_X=\alpha\sum_i x_i^4+\beta$.
Hence constancy on $S^{d-1}$ is equivalent to $\alpha=0$, and it suffices to compare $F_X$ at two points
where $\sum_i x_i^4$ takes different values.

\begin{proposition}
\label{prop:two-point-test}
Assume that $X\subset G_{2,d}$ is $B_d$-invariant. Then the following are equivalent:
\begin{enumerate}
\item[\rm (i)] $X$ is a $\TFF_2$.
\item[\rm (ii)]
\[
F_X(e_1)=F_X\Bigl(\frac{e_1+e_2}{\sqrt2}\Bigr).
\]
\end{enumerate}
\end{proposition}

\begin{proof}
By Lemma~\ref{lem:bd-inv-form},
\[
F_X(x)=\alpha\sum_i x_i^4+\beta\Bigl(\sum_i x_i^2\Bigr)^2.
\]
On the sphere, $(\sum_i x_i^2)^2=1$, hence
$F_X|_{S^{d-1}}=\alpha\sum_i x_i^4+\beta$.
Thus $F_X$ is constant on $S^{d-1}$ if and only if $\alpha=0$.

On the other hand,
\[
\sum_i (e_{1}^4)_i=1,\qquad
\sum_i\Bigl(\frac{e_1+e_2}{\sqrt2}\Bigr)_i^4=\frac12,
\]
and therefore
\[
F_X(e_1)-F_X\Bigl(\frac{e_1+e_2}{\sqrt2}\Bigr)
=\alpha\Bigl(1-\frac12\Bigr)=\frac{\alpha}{2}.
\]
Hence (ii) $\iff \alpha=0 \iff$ (i).
\end{proof}

Next we define explicit $B_d$-invariant sets as $B_d$-orbits.
Let $a$ and $b$ be integers with $1\le a\le b$, $a+b\le d$, and write $[d]:=\{1,\dots,d\}.$
Define
\[
u_0:=\frac1{\sqrt a}\sum_{i=1}^a e_i,\qquad
v_0:=\frac1{\sqrt b}\sum_{j=a+1}^{a+b} e_j,\qquad
W_0:=\mathrm{span}(u_0,v_0)\in G_{2,d}.
\]

\[
\mathcal O^{(d)}_{a,b}
:=B_d\cdot W_0\subset G_{2,d},
\]

\[
N_d(a,b):=\bigl|\mathcal O^{(d)}_{a,b}\bigr|.
\]

To compute $N_d(a,b)$ and, later, the quantity $F_{\mathcal O}(x)$,
we introduce the following set which will be used to parametrize elements of $\mathcal O^{(d)}_{a,b}$ by supports and signs:
\[
\mathcal{T}_{a,b}^{(d)}
:=
\Bigl\{(A,B,\varepsilon,\delta)\ \Big|\ 
A,B\subset[d],\ A\cap B=\varnothing,\ |A|=a,\ |B|=b,\ 
\varepsilon\in\{\pm1\}^{A},\ \delta\in\{\pm1\}^{B}
\Bigr\},
\]
where $\{\pm1\}^{A}$ denotes the set of all functions $A\to\{\pm1\}$.

For \(\tau=(A,B,\varepsilon,\delta)\in\mathcal{T}_{a,b}^{(d)}\), set
\[
u_{A,\varepsilon}:=\frac{1}{\sqrt{a}}\sum_{i\in A}\varepsilon_i\, e_i,\qquad
v_{B,\delta}:=\frac{1}{\sqrt{b}}\sum_{j\in B}\delta_j\, e_j,
\]
and define
\[
W_{\tau}=W_{A,B;\varepsilon,\delta}:=\mathrm{span}\bigl(u_{A,\varepsilon},\,v_{B,\delta}\bigr)\in G_{2,d}.
\]


\begin{lemma}\label{lem:Wtau_in_orbit}
The map defined by
\[
\pi:\mathcal{T}_{a,b}^{(d)}\longrightarrow \mathcal{O}_{a,b}^{(d)},
\qquad
\pi(\tau):=W_{\tau}.
\]
is surjective.
\end{lemma}

\begin{proof}
We first show that \(\pi\) is well-defined, i.e., \(W_{\tau}\in \mathcal O^{(d)}_{a,b}\) for every
\(\tau=(A,B,\varepsilon,\delta)\in\mathcal T^{(d)}_{a,b}\).
Choose a permutation \(\sigma\in S_d\) such that
\(\sigma(\{1,\dots,a\})=A\) and \(\sigma(\{a+1,\dots,a+b\})=B\).
Let \(s\in\{\pm1\}^d\) be a sign vector such that \(s_i=\varepsilon_i\) for \(i\in A\),
\(s_j=\delta_j\) for \(j\in B\), and \(s_\ell=1\) otherwise, and set \(g:=\mathrm{diag}(s)\,\sigma\in B_d\).
Then
\[
g u_0=\frac1{\sqrt a}\sum_{i\in A}\varepsilon_i e_i = u_{A,\varepsilon},
\qquad
g v_0=\frac1{\sqrt b}\sum_{j\in B}\delta_j e_j = v_{B,\delta},
\]
and hence \(gW_0=\mathrm{span}(g u_0,g v_0)=\mathrm{span}(u_{A,\varepsilon},v_{B,\delta})=W_{\tau}\).
Therefore \(W_{\tau}\in B_d\cdot W_0=\mathcal O^{(d)}_{a,b}\), so \(\pi\) is well-defined.

To show that $\pi$ is surjective, let \(W\in\mathcal O^{(d)}_{a,b}\). Then \(W=gW_0\) for some \(g\in B_d\).
Write \(g=\mathrm{diag}(s)\,\sigma\) with \(s\in\{\pm1\}^d\) and \(\sigma\in S_d\), and put
\[
A:=\sigma(\{1,\dots,a\}),\qquad B:=\sigma(\{a+1,\dots,a+b\}),
\qquad
\varepsilon_i:=s_i\ (i\in A),\quad \delta_j:=s_j\ (j\in B).
\]
Then \(\tau:=(A,B,\varepsilon,\delta)\in\mathcal T^{(d)}_{a,b}\) and the same computation gives
\(g u_0=u_{A,\varepsilon}\) and \(g v_0=v_{B,\delta}\), hence
\(W=gW_0=W_{\tau}=\pi(\tau)\).
Thus, \(\pi\) is surjective.
\end{proof}

For $x=(x_1,\dots,x_d)\in\R^d$, define its support by
\[
\supp(x):=\{\,i\in[d]:x_i\neq 0\,\}.
\]

\begin{proposition}\label{prop:inverse}
Let $W \in \mathcal O^{(d)}_{a,b}$. Then
\[
|\pi^{-1}(W)|=
\begin{cases}
4,&\text{if } a<b,\\
8,&\text{if } a=b.
\end{cases}
\]
\end{proposition}

\begin{proof}
Fix $W\in\mathcal O^{(d)}_{a,b}$ and, by Lemma~\ref{lem:Wtau_in_orbit}, there exists $\tau=(A,B,\varepsilon,\delta)\in\mathcal T^{(d)}_{a,b}$ such that $W=W_{\tau}$. Write $u:=u_{A,\varepsilon}$, $v:=v_{B,\delta}$.
Since $A\cap B=\varnothing$, we have $\supp(u)=A$, $\supp(v)=B$, and $u\perp v$.
For any $x=\alpha u+\beta v\in W\setminus\{0\}$,
\[
\supp(x)=
\begin{cases}
A,&\text{if } \beta=0,\\
B,&\text{if } \alpha=0,\\
A\cup B,&\text{if } \alpha\beta\neq 0,
\end{cases}
\qquad\text{hence}\qquad
|\supp(x)|\in\{a,b,a+b\}.
\tag{$\ast$}\label{eq:supp-pattern-en}
\]
In particular, the vectors in $W$ satisfying $|\supp(x)|=a$ lie on the line $\R u$,
and their support is always $A$ (and similarly for $|\supp(x)|=b$).

Now assume $\tau'=(A',B',\varepsilon',\delta')$ satisfies $W_{(\tau')}=W$.
If $a<b$, then the line in $U$ spanned by vectors with support size $a$ is unique, so
$\R u_{A,\varepsilon}=\R u_{A',\varepsilon'}$, hence $A=A'$.
Similarly $B=B'$.
If $a=b$, the same argument shows $\{A,B\}=\{A',B'\}$, leaving only the possibility of swapping $A$ and $B$.

Thus, when $a<b$, the pair $(A,B)$ determining $W$ is unique, whereas when $a=b$ there are exactly two choices,
namely $(A,B)$ and $(B,A)$.
Hence it remains to count the identifications coming from signs for fixed $(A,B)$.

Fix $(A,B)$ and consider
\[
W_{A,B;\varepsilon,\delta}=W_{A,B;\varepsilon',\delta'}
\quad\Longleftrightarrow\quad
\varepsilon'=\pm\varepsilon\ \text{and}\ \delta'=\pm\delta.
\label{eq:sign-identification-en}
\]
We prove ($\Rightarrow$). If $W_{A,B;\varepsilon',\delta'}=W_{A,B;\varepsilon,\delta}$, then
$u_{A,\varepsilon'}\in \Span\{u_{A,\varepsilon},v_{B,\delta}\}$, so for some $\alpha,\beta\in\R$ we have
$u_{A,\varepsilon'}=\alpha u_{A,\varepsilon}+\beta v_{B,\delta}$.
The left-hand side is supported on $A$ and $v_{B,\delta}$ is supported on $B$, so $A\cap B=\varnothing$ forces $\beta=0$.
Thus $u_{A,\varepsilon'}=\alpha u_{A,\varepsilon}$.
Since both are unit vectors, $|\alpha|=1$, hence $\alpha=\pm1$, which is equivalent to $\varepsilon'=\pm\varepsilon$.
Similarly $\delta'=\pm\delta$ follows. The converse ($\Leftarrow$) is immediate.

Therefore, for fixed $(A,B)$ there are exactly $4$ sign classes.
Hence $|\pi^{-1}(W)|=4$ when $a<b$.
When $a=b$, we can swap $(A,B)$, hence $|\pi^{-1}(W)|=8$.
\end{proof}

\begin{lemma}\label{lem:disjoint-and-size}
Let $d\in\N$, and let $1\le a\le b$ with $a+b\le d$, and consider the $B_d$-orbit
\(
\mathcal{O}^{(d)}_{a,b}\subset G_{2,d}.
\)
\begin{enumerate}
\item[\rm (i)]
Let $1\le a'\le b'$ with $a'+b'\le d$.
If $\{a,b\}\neq\{a',b'\}$, then
\[
\mathcal{O}^{(d)}_{a,b}\cap \mathcal{O}^{(d)}_{a',b'}=\varnothing.
\]
\item[\rm (ii)]
The orbit size is given by
\[
N_d(a,b)
=\binom{d}{a}\binom{d-a}{b}\,2^{a+b-2}\times
\begin{cases}
1,&\text{if } a<b,\\[1mm]
\frac{1}{2},&\text{if } a=b.
\end{cases}
\]
\end{enumerate}
\end{lemma}

\begin{proof}
(i)
Assume $W\in \mathcal O^{(d)}_{a,b}$. By Lemma~\ref{lem:Wtau_in_orbit}, there exists
$\tau=(A,B,\varepsilon,\delta)\in\mathcal T^{(d)}_{a,b}$ such that
\[
W=\pi(\tau)=W_{A,B;\varepsilon,\delta}=\Span\{u_{A,\varepsilon},v_{B,\delta}\}.
\]
Put $u:=u_{A,\varepsilon}$ and $v:=v_{B,\delta}$. Then $\supp(u)=A$, $\supp(v)=B$, $A\cap B=\varnothing$,
and $|A|=a$, $|B|=b$. Hence for any $x=\alpha u+\beta v\neq 0$,
\[
|\supp(x)|=
\begin{cases}
a,&\beta=0,\\
b,&\alpha=0,\\
a+b,&\alpha\beta\neq 0.
\end{cases}
\]

Define
\[
m_1(W):=\min\{|\supp(x)|:x\in W\setminus\{0\}\},
\]
and choose $x_1\in W\setminus\{0\}$ attaining $m_1(W)$. Then $m_1(W)=a$.
Set
\[
m_2(W):=\min\{|\supp(x)|:x\in W\setminus(\R x_1)\}.
\]
Note that $m_2(W)$ is independent of the choice of $x_1$:
if $a<b$, then $|\supp(x)|=a$ forces $\beta=0$, so $x_1\in\R u$ and hence $m_2(W)=b$;
if $a=b$, then $m_1(W)=m_2(W)=a$ for any such $x_1$.
Therefore the unordered pair $\{a,b\}$ is determined by $W$.

Consequently, if $W$ belonged to both $\mathcal O^{(d)}_{a,b}$ and $\mathcal O^{(d)}_{a',b'}$, then we would have
$\{a,b\}=\{a',b'\}$. Thus, if $\{a,b\}\neq\{a',b'\}$, the two orbits are disjoint.

(ii)
We have
\[
|\mathcal T_{a,b}^{(d)}|
=\binom{d}{a}\binom{d-a}{b}\cdot 2^a\cdot 2^b
=\binom{d}{a}\binom{d-a}{b}\,2^{a+b}.
\]
On the other hand, Proposition~\ref{prop:inverse} shows that $|\pi^{-1}(W)|$ is independent of $W\in\mathcal O^{(d)}_{a,b}$.
Hence
\[
N_d(a,b)=|\mathcal O^{(d)}_{a,b}|
=\frac{|\mathcal T_{a,b}^{(d)}|}{|\pi^{-1}(W)|}
=\binom{d}{a}\binom{d-a}{b}\,2^{a+b}\times
\begin{cases}
\frac14,&\text{if } a<b,\\[1mm]
\frac18,&\text{if } a=b,
\end{cases}
\]
which simplifies to the stated formula.
\end{proof}

By Proposition~\ref{prop:two-point-test}, a $B_d$-invariant set $X$ is a $\TFF_2$ if and only if $F_X$
takes the same value at two points. We measure this discrepancy by the following quantity.

For an orbit $X \subset G_{2,d}$, define
\[
\Delta(X):=
F_{X}(e_1)
-
F_{X}\Bigl(\frac{e_1+e_2}{\sqrt2}\Bigr).
\]
\begin{theorem}
\label{thm:orbit-union}
Let $(a_t,b_t)\in\mathbb N^2$ ($t=1,\dots,m$) be pairwise distinct pairs satisfying $1\le a_t\le b_t$ and $a_t+b_t\le d$, and set
\[
X:=\bigsqcup_{t=1}^m \mathcal O^{(d)}_{a_t,b_t}\subset G_{2,d}.
\]
Then the following are equivalent:
\begin{enumerate}
\item[\rm (i)] $X$ is a $\TFF_2$.
\item[\rm (ii)]
\[
\sum_{t=1}^m N_d(a_t,b_t)\,\Delta\!\bigl(\mathcal O^{(d)}_{a_t,b_t}\bigr)=0.
\]
\end{enumerate}
\end{theorem}

\begin{proof}
Each \(\mathcal O^{(d)}_{a_t,b_t}\) is a \(B_d\)-orbit and hence $B_d$-invariant, so their disjoint union
\(X=\bigsqcup_{t=1}^m \mathcal O^{(d)}_{a_t,b_t}\) is also $B_d$-invariant.
Therefore, by Proposition~\ref{prop:two-point-test},
\begin{equation}\label{eq:two-point-en}
X\ \text{is a $\TFF_2$}
\quad\Longleftrightarrow\quad
F_X(e_1)=F_X\Bigl(\frac{e_1+e_2}{\sqrt2}\Bigr).
\end{equation}

On the other hand, $|X|=\sum_{t=1}^m N_d(a_t,b_t)$, and
\[
F_X(x)
=\frac1{|X|}\sum_{W\in X}\|P_Wx\|^4
=\frac1{\sum_{t=1}^m N_d(a_t,b_t)}
\sum_{t=1}^m N_d(a_t,b_t)\,F_{\mathcal O^{(d)}_{a_t,b_t}}(x).
\]
Applying this to $x=e_1$ and $x=(e_1+e_2)/\sqrt2$ and taking the difference yields
\[
\Delta(X)
=\frac1{\sum_{t=1}^m N_d(a_t,b_t)}
\sum_{t=1}^m N_d(a_t,b_t)\,\Delta\!\bigl(\mathcal O^{(d)}_{a_t,b_t}\bigr).
\]
Since the denominator is positive, combining this with \eqref{eq:two-point-en} gives the claim.
\end{proof}

In particular, for a single orbit $X=\mathcal O^{(d)}_{a,b}$, the $\TFF_2$ condition is equivalent to $\Delta(X)=0$.
Thus, it suffices to compute $\Delta(\mathcal O^{(d)}_{a,b})$ explicitly.

\begin{lemma}
\label{lem:orbit-two-values}
Let $a,b,d\in\N$ with $a+b\le d$. Then
\begin{equation}\label{eq:Fx_e1}
F_{\mathcal O^{(d)}_{a,b}}(e_1)=\frac{a+b}{ab\,d}.
\end{equation}
\begin{equation}\label{eq:Fx_e12}
F_{\mathcal O^{(d)}_{a,b}}\Bigl(\frac{e_1+e_2}{\sqrt2}\Bigr)
=\frac{8ab+(d-4)(a+b)}{2ab\,d(d-1)}.
\end{equation}
Moreover,
\begin{equation}\label{eq:Delta_orbit}
\Delta\bigl(\mathcal O^{(d)}_{a,b}\bigr)
=
\frac{(d+2)(a+b)-8ab}{2ab\,d(d-1)}.
\end{equation}
\end{lemma}

\begin{proof}
Write \(\mathcal O:=\mathcal O^{(d)}_{a,b}\) and \(\mathcal T:=\mathcal T^{(d)}_{a,b}\).
By Proposition~\ref{prop:inverse}, for any \(x\in\mathbb R^d\) we have
\begin{equation}\label{eq:pullback-orbit-average-en}
F_{\mathcal O}(x)
=\frac1{|\mathcal O|}\sum_{W\in\mathcal O}\|P_Wx\|^4
=\frac1{|\mathcal T|}\sum_{\tau\in\mathcal T}\|P_{W_{\tau}}x\|^4.
\end{equation}
Moreover,
\begin{equation}\label{eq:size-T-en}
|\mathcal T|=\binom{d}{a}\binom{d-a}{b}\,2^{a+b}.
\end{equation}
We evaluate \eqref{eq:pullback-orbit-average-en} at the two vectors
\[
x_1:=e_1,
\qquad
x_2:=\frac{e_1+e_2}{\sqrt2}.
\]

Fix \(\tau=(A,B,\varepsilon,\delta)\in\mathcal T\), and write
\(u:=u_{A,\varepsilon}\), \(v:=v_{B,\delta}\), and \(W_{\tau}=\mathrm{span}(u,v)\).
Since \(A\cap B=\varnothing\), we have \(u\perp v\) and \(\|u\|=\|v\|=1\), hence
\begin{equation}\label{eq:proj-expansion-en}
\|P_{W_{\tau}}x\|^2=\langle x,u\rangle^2+\langle x,v\rangle^2,\qquad
\|P_{W_{\tau}}x\|^4
=\langle x,u\rangle^4+\langle x,v\rangle^4
+2\langle x,u\rangle^2\langle x,v\rangle^2.
\end{equation}

\noindent\textbf{1. Proof of \eqref{eq:Fx_e1}.}
We have
\[
\langle e_1,u\rangle=
\begin{cases}
\varepsilon_1/\sqrt a,&\text{if } 1\in A,\\
0,&\text{if } 1\notin A,
\end{cases}
\qquad
\langle e_1,v\rangle=
\begin{cases}
\delta_1/\sqrt b,&\text{if } 1\in B,\\
0,&\text{if } 1\notin B.
\end{cases}
\]
So
\[
\|P_{W_{\tau}}e_1\|^4=
\begin{cases}
1/a^2,&\text{if } 1\in A,\\
1/b^2,&\text{if } 1\in B,\\
0,&\text{if } 1\notin A\cup B.
\end{cases}
\]
Therefore,
\begin{align*}
\sum_{\tau\in\mathcal T}\|P_{W_{\tau}}e_1\|^4
&=\frac1{a^2}\,|\{\tau\in\mathcal T:1\in A\}|+\frac1{b^2}\,|\{\tau\in\mathcal T:1\in B\}|,\\
|\{\tau\in\mathcal T:1\in A\}|
&=\binom{d-1}{a-1}\binom{d-a}{b}\,2^{a+b},\\
|\{\tau\in\mathcal T:1\in B\}|
&=\binom{d-1}{b-1}\binom{d-b}{a}\,2^{a+b}.
\end{align*}
Combining this with \eqref{eq:pullback-orbit-average-en} and \eqref{eq:size-T-en} yields
\[
F_{\mathcal O}(e_1)
=\frac1{a^2}\cdot\frac{\binom{d-1}{a-1}}{\binom{d}{a}}
+\frac1{b^2}\cdot\frac{\binom{d-1}{b-1}}{\binom{d}{b}}
=\frac{a+b}{ab\,d}.
\]

\noindent\textbf{2. Proof of \eqref{eq:Fx_e12}.}

Fix \(\tau=(A,B,\varepsilon,\delta)\in\mathcal T\), and write
\(u:=u_{A,\varepsilon}\), \(v:=v_{B,\delta}\), and \(W_{\tau}=\mathrm{span}(u,v)\).
Set $x:=x_2$.

{(AA) The case $\{1,2\}\subset A$.}
Then $\langle x,v\rangle=0$ and
\(
\langle x,u\rangle=(\varepsilon_1+\varepsilon_2)/\sqrt{2a}.
\)
Hence
\[
\sum_{\varepsilon\in\{\pm1\}^{A}}(\varepsilon_1+\varepsilon_2)^4
=2^{a-2}\!\!\sum_{(\varepsilon_1,\varepsilon_2)\in\{\pm1\}^2}\!(\varepsilon_1+\varepsilon_2)^4
=32\cdot 2^{a-2}.
\]
Thus for fixed \((A,B)\),
\[
\sum_{\varepsilon}\sum_{\delta}\|P_{W_{\tau}}x\|^4
=2^b\cdot\frac{32\cdot 2^{a-2}}{4a^2}
=\frac{2^{a+b+1}}{a^2},
\]
and the number of such pairs \((A,B)\) is \(\binom{d-2}{a-2}\binom{d-a}{b}\).

{(BB) The case $\{1,2\}\subset B$.}
Similarly,
\[
\sum_{\varepsilon}\sum_{\delta}\|P_{W_{\tau}}x\|^4=\frac{2^{a+b+1}}{b^2},
\qquad
\#\{(A,B)\}=\binom{d-2}{a}\binom{d-a-2}{b-2}.
\]

{(AB) The case $\{1,2\}$ is split, one element in $A$ and one in $B$.}
For instance, if \(1\in A\) and \(2\in B\), then
\(
\langle x,u\rangle=\varepsilon_1/\sqrt{2a}
\)
and
\(
\langle x,v\rangle=\delta_2/\sqrt{2b}.
\)
By \eqref{eq:proj-expansion-en},
\[
\|P_{W_{\tau}}x\|^4=\Bigl(\frac{1}{2a}+\frac{1}{2b}\Bigr)^2
=\frac{1}{4a^2}+\frac{1}{4b^2}+\frac{1}{2ab},
\]
and hence
\[
\sum_{\varepsilon}\sum_{\delta}\|P_{W_{\tau}}x\|^4
=2^{a+b}\Bigl(\frac{1}{4a^2}+\frac{1}{4b^2}+\frac{1}{2ab}\Bigr).
\]
The number of such pairs \((A,B)\) equals
\(2\binom{d-2}{a-1}\binom{d-a-1}{b-1}\).

{(AC) The case one of $\{1,2\}$ lies in $A$ and the other lies in $[d]\setminus(A\cup B)$.}
For instance, if \(1\in A\) and \(2\notin A\cup B\), then
\(
\langle x,u\rangle=\varepsilon_1/\sqrt{2a}
\)
and \(\langle x,v\rangle=0\), so
\[
\sum_{\varepsilon}\sum_{\delta}\|P_{W_{\tau}}x\|^4
=2^{a+b}\cdot\frac{1}{4a^2},
\qquad
\#\{(A,B)\}=2\binom{d-2}{a-1}\binom{d-a-1}{b}.
\]

{(BC) The case one of $\{1,2\}$ lies in $B$ and the other lies in $[d]\setminus(A\cup B)$.}
Similarly,
\[
\sum_{\varepsilon}\sum_{\delta}\|P_{W_{\tau}}x\|^4
=2^{a+b}\cdot\frac{1}{4b^2},
\qquad
\#\{(A,B)\}=2\binom{d-2}{a}\binom{d-a-2}{b-1}.
\]
In the remaining case \(\{1,2\}\cap(A\cup B)=\varnothing\), we have \(\|P_{W_{\tau}}x\|^4=0\).

Substituting these contributions into \eqref{eq:pullback-orbit-average-en} and using \eqref{eq:size-T-en}, we obtain
\[
F_{\mathcal O}(x)
=\frac{8ab+(d-4)(a+b)}{2ab\,d(d-1)}.
\]

Finally, combining the results of 1 and 2 yields
\[
\Delta\bigl(\mathcal O^{(d)}_{a,b}\bigr)
=
\frac{a+b}{ab\,d}
-
\frac{8ab+(d-4)(a+b)}{2ab\,d(d-1)}
=
\frac{(d+2)(a+b)-8ab}{2ab\,d(d-1)}.
\]
\end{proof}

\begin{theorem}\label{thm:scaling}
Assume that $(d_0;a_0,b_0)\in\mathbb N^3$ satisfies
\[
(d_0+2)(a_0+b_0)=8a_0b_0,\qquad 1\le a_0\le b_0,\qquad a_0+b_0\le d_0.
\]
For any integer \(s\ge1\), set
\[
d:=(d_0+2)s-2,\qquad a:=a_0s,\qquad b:=b_0s.
\]
Then
\[
X:=\mathcal O^{(d)}_{a,b}\subset G_{2,d}
\]
is an equal-weight $\TFF_2$.
\end{theorem}

\begin{proof}
By Lemma~\ref{lem:orbit-two-values}, the single orbit $X=\mathcal O^{(d)}_{a,b}$ is a $\TFF_2$ if and only if
\[
\Delta(X)=\Delta\bigl(\mathcal O^{(d)}_{a,b}\bigr)=0
\quad\Longleftrightarrow\quad
(d+2)(a+b)=8ab.
\]
On the other hand,
\[
d+2=(d_0+2)s,\qquad a+b=(a_0+b_0)s,\qquad ab=a_0b_0\,s^2,
\]
and hence
\[
(d+2)(a+b)= (d_0+2)(a_0+b_0)\,s^2 = 8a_0b_0\,s^2 = 8ab.
\]
Thus $\Delta(X)=0$, proving the claim. Moreover,
\[
d-a-b = (d_0+2-a_0-b_0)s-2\ge 2s-2\ge 0
\]
follows from $a_0+b_0\le d_0$.
\end{proof}

\begin{remark}
\label{rem:minimal-two-families}
A sufficient condition for a single orbit \(X=\mathcal O^{(d)}_{a,b}\subset G_{2,d}\) to be a $\TFF_2$ is
\[
(d+2)(a+b)=8ab,\qquad 1\le a\le b,\qquad a+b\le d
\]
(cf.\ Lemma~\ref{lem:orbit-two-values} and Theorem~\ref{thm:orbit-union} with \(m=1\)).
We record two explicit infinite families satisfying this condition.

\textbf{(1)} The triple \((d_0;a_0,b_0)=(4;1,3)\) satisfies
\(
(4+2)(1+3)=8\cdot1\cdot3.
\)
For any integer \(s\ge1\), set
\[
d=6s-2,\qquad a=s,\qquad b=3s,\qquad d-a-b=2s-2\ge0.
\]
Then, by Theorem~\ref{thm:scaling},
\[
X_s:=\mathcal O^{(6s-2)}_{s,\,3s}\subset G_{2,\,6s-2}
\]
is a single-orbit $\TFF_2$. The first term is \((d;a,b)=(4;1,3)\).

\textbf{(2)} The triple \((d_0;a_0,b_0)=(13;3,5)\) satisfies
\(
(13+2)(3+5)=8\cdot3\cdot5.
\)
To obtain only odd dimensions, set \(s=2r+1\ (r\ge0)\), and define
\[
d=15(2r+1)-2=30r+13,\quad
a=3(2r+1)=6r+3,\quad
b=5(2r+1)=10r+5.
\]
Then, by Theorem~\ref{thm:scaling},
\[
Y_r:=\mathcal O^{(30r+13)}_{\,6r+3,\,10r+5}\subset G_{2,\,30r+13}
\]
is a single-orbit $\TFF_2$. The first term is \((d;a,b)=(13;3,5)\).

By Lemma~\ref{lem:disjoint-and-size}\,(ii), when \(a<b\) we have
\[
\bigl|\mathcal O^{(d)}_{a,b}\bigr|
= N_d(a,b)
=\binom{d}{a}\binom{d-a}{b}\,2^{a+b-2}.
\]
Hence the sizes of these families are
\[
|X_s|
= N_{6s-2}(s,3s)
=\binom{6s-2}{s}\binom{5s-2}{3s}\,2^{4s-2},
\]
\[
|Y_r|
= N_{30r+13}(6r+3,10r+5)
=\binom{30r+13}{6r+3}\binom{24r+10}{10r+5}\,2^{16r+6}.
\]
\end{remark}

The following list arises from a brute-force search: for each odd $d\le 1000$, we exhaustively searched over
all pairs $(a,b)$ with $1\le a\le b$ and $a+b\le d$, and tested whether there exist two pairs
\((a_1,b_1)\), \((a_2,b_2)\) satisfying
\[
N_d(a_1,b_1)\Delta(a_1,b_1)+N_d(a_2,b_2)\Delta(a_2,b_2)=0.
\]

\begin{remark}
\label{rem:two-orbit-search}
Computer search shows that, for the following odd dimensions \(d\le 1000\), there exists at least one solution
given by a union of two orbits:
\[
\begin{aligned}
&5,\,7,\,13,\,19,\,33,\,93,\,133,\,163,\,193,\,229,\,253,\,283,\,313,\,349,\,403,\,427,\,493,\,523,\,583,\,613,\\
&661,\,673,\,691,\,713,\,733,\,739,\,763,\,817,\,823,\,853,\,933,\,943,\,973.
\end{aligned}
\]
As small examples, we obtain
\[
d=5:\ (a,b)=(1,1)\ \sqcup\ (2,2),\qquad
d=7:\ (a,b)=(1,3)\ \sqcup\ (3,3).
\]
\end{remark}

\section{Non-existence problem for certain tight $2$-fusion frames on $G_{2,d}$}
\label{sec:NP-ECTFF2-G2d}

In this section, we focus primarily on \emph{equi-chordal tight $2$-fusion frames} ($\ECTFF_2$) on $G_{2,d}$,
and derive necessary conditions for their existence, including upper/lower bounds on the cardinality and the corresponding equality cases.
In the context of Grassmannian packings, Conway--Hardin--Sloane~\cite{CHS1996} introduced an isometric embedding into a sphere via projection matrices,
and obtained the simplex bound for equi-chordal (EC) configurations.
Using this embedding, we also summarize the relationship between design properties on the Grassmann manifold and those on the sphere.

For $V\in G_{k,d}$, let $P_V$ denote the orthogonal projection onto $V$.
Let $\Sym(d)$ denote the space of real symmetric $d\times d$ matrices, and set
\[
H_0:=\{A\in\Sym(d):\tr(A)=0\},\qquad
D_0:=\dim H_0=\binom{d+1}{2}-1.
\]
Equip $H_0$ with the Frobenius inner product
\[
\langle A,B\rangle_F:=\tr(AB)\qquad (A,B\in H_0).
\]
Define
\[
R_V:=\sqrt{\frac{d}{k(d-k)}}\Bigl(P_V-\frac{k}{d}I_d\Bigr)\in H_0,
\qquad
\Phi:G_{k,d}\to S^{D_0-1}\subset H_0,\quad V\mapsto R_V.
\]
\begin{theorem}[Conway--Hardin--Sloane {\cite[Theorem~5.1]{CHS1996}}]
With the above notation, $\Phi$ is an isometric embedding with respect to the chordal distance.
In particular, if a finite set $\mathcal{D}\subset G_{k,d}$ is equi-chordal, then
$\Phi(\mathcal{D})\subset S^{D_0-1}$ is an equidistant set on the sphere.
\end{theorem}

\begin{corollary}[Conway--Hardin--Sloane {\cite[Corollary~5.2]{CHS1996}}]\label{cor:simplex}
For $\{V_i\}_{i=1}^N \subset G_{k,d}$, the minimal chordal distance $d_C$ satisfies
\[
d_C^{\,2}\ \le\ \frac{k(d-k)}{d}\cdot\frac{N}{N-1}.
\]
Moreover, equality requires $N\le D_0+1$, and equality holds if and only if
$\Phi(V_i)\in H_0\simeq\R^{D_0}$ satisfy
\[
\sum_{i=1}^N \Phi(V_i)=0,\qquad
\langle \Phi(V_i),\Phi(V_j)\rangle_F=-\frac{1}{N-1}\ (i\ne j),
\]
i.e.\ $\{\Phi(V_i)\}$ forms a regular simplex in $S^{D_0-1}$.
\end{corollary}

In what follows, we show that $\Phi$ sends a Grassmann $4$-design to a spherical $2$-design.


\begin{lemma}\label{lem:GtoS}
Let $X \subset G_{k,d}$ be a Grassmann $4$-design.
Then $Y:=\Phi(X)$ is a spherical $2$-design on $S^{D_0-1}$.
\end{lemma}

\begin{proof}
By Proposition~\ref{prop:sph-design-equiv}, it suffices to show that
\[
\sum_{V,W\in X} Q_{\ell}^{(D_0)}(\langle R_V,R_W\rangle_F)=0
\qquad (\ell=1,2).
\]
By the definition of $R_V$,
\[
\langle R_V,R_W\rangle_F=\frac{d}{k(d-k)}\left(\tr(P_VP_W)-\frac{k^2}{d}\right).
\]
Also $\tr(P_VP_W)=\tr(AA^{\mathsf T})=\sum_{i=1}^k\sigma_i^2=\sum_{i=1}^k y_i(V,W)$, so
$c_2=\frac1k\tr(P_VP_W)$.
Substituting this into $P_{(2)}=\frac{k}{k-d}-\frac{d}{k-d}c_2$ yields
\[
P_{(2)}(V,W)=\frac{d}{k(d-k)}\left(\tr(P_VP_W)-\frac{k^2}{d}\right)
=\langle R_V,R_W\rangle_F.
\]
Hence the case $\ell=1$ is exactly $\sum_{V,W\in X}P_{(2)}(V,W)=0$,
which follows from the assumption that $X$ is a Grassmann $4$-design (and thus a $2$-design).

For $\ell=2$, using
\[
Q_2^{(D_0)}(x)=\frac{D_0x^2-1}{D_0-1},
\]
we obtain
\[
\sum_{V,W\in X} Q_{2}^{(D_0)}(\langle R_V,R_W\rangle_F)
=\frac{D_0}{D_0-1}\sum_{V,W\in X} P_{(2)}(V,W)^2-\frac{|X|^2}{D_0-1}.
\]
Here we may expand
\[
\begin{aligned}
P_{(2)}(V,W)^2
&=
\frac{1}{D_0}
+\frac{4(d-2k)^2}{\,k(d-2)(d+4)(d-k)\,}\,P_{(2)}(V,W)\\[0.6ex]
&\quad
+\frac{d^2(k+2)(d-k+2)}{\,3k(d+2)(d+4)(d-k)\,}\,P_{(4)}(V,W)
+\frac{2d^2(k-1)(d-k-1)}{\,3k(d-2)(d-1)(d-k)\,}\,P_{(2,2)}(V,W).
\end{aligned}
\]
Since $X$ is a Grassmann $4$-design, we have
\[
\sum_{V,W\in X}P_{(2)}(V,W)=
\sum_{V,W\in X}P_{(4)}(V,W)=
\sum_{V,W\in X}P_{(2,2)}(V,W)=0,
\]
and hence
\[
\sum_{V,W\in X}P_{(2)}(V,W)^2=\frac{|X|^2}{D_0}.
\]
Substituting this shows that the $\ell=2$ sum also vanishes.
\end{proof}

The number $\binom{d+1}{2}$ is the absolute lower bound for Grassmann $4$-designs on $G_{k,d}$ (see e.g.\ \cite{BCN2002}).

\begin{definition}\cite{BBC2004}\label{def:tight4}
Let $d\ge4$.
A finite set $X\subset G_{k,d}$ is called a \emph{tight Grassmann $4$-design} if
$X$ is a Grassmann $4$-design and
\[
|X|=\binom{d+1}{2}.
\]
\end{definition}

\begin{corollary}\label{cor:GtoS}
Let $X \subset G_{k,d}$ be a tight Grassmann $4$-design.
Then $X$ is equi-chordal.
\end{corollary}

\begin{proof}
If $X$ is a tight Grassmann $4$-design, then $|X|=\binom{d+1}{2}=D_0+1$.
By Lemma~\ref{lem:GtoS}, $Y=\Phi(X)$ is a spherical $2$-design on $S^{D_0-1}$ with $|Y|=D_0+1$.
Hence $\sum_{y\in Y}y=0$, so $Y$ achieves the equality in Corollary~\ref{cor:simplex}.
Therefore $Y$ forms a regular simplex in $S^{D_0-1}$, and in particular is equidistant.
Since $\Phi$ is an isometric embedding with respect to the chordal distance, $X$ is equi-chordal.
\end{proof}

The following theorem provides basic cardinality constraints for $\ECTFF_2$ in $G_{2,d}$.

\begin{theorem}\label{thm:ECTFF_2}
Let $d\ge 4$, and let $X=\{V_i\}_{i=1}^{N}\subset G_{2,d}$ be an $\ECTFF_2$.
Then
\[
\frac{d^2}{4} \le N \le \binom{d+1}{2}.
\]
Moreover, equality in the lower bound holds if and only if $X$ is an $\EITFF_2$ on $G_{2,d}$,
and equality in the upper bound holds if and only if $X$ is a tight $4$-design on $G_{2,d}$.
\end{theorem}
\begin{proof}
For $V,W\in G_{2,d}$, set
\begin{equation*}
e_1(V,W):=y_1(V,W)+y_2(V,W),\qquad
e_2(V,W):=y_1(V,W)y_2(V,W).
\end{equation*}

Since $X$ is equi-chordal, $e_1(V_i,V_j)$ is constant for $i\neq j$. Set
\begin{equation*}
e_{1,0}:=e_1(V_i,V_j)\qquad (i\neq j).
\end{equation*}
Also define
\begin{equation*}
\overline{e_2(N)}:=\frac{1}{N^2-N}\sum_{i\neq j} e_2(V_i,V_j).
\end{equation*}

\textbf{(i) Lower bound.}
Since $X$ is an equal-weight tight $2$-fusion frame, Theorem~\ref{thm:BE-cubature}{\rm(iii)} applied with $t=2$
(and weights $\omega_i=1/N$) yields the moment identities for the one-row zonal polynomials $P_{(2)}$ and $P_{(4)}$:

\begin{equation}\label{eq:tff2-moments}
\sum_{i,j=1}^N P_{(2)}(V_i,V_j)=0,\qquad
\sum_{i,j=1}^N P_{(4)}(V_i,V_j)=0.
\end{equation}
By \eqref{eq:P_2} (with $k=2$), we obtain
\begin{equation*}
P_{(2)}(V,W)=\frac{4-d\,e_1(V,W)}{2(2-d)}.
\end{equation*}
Since $e_1(V_i,V_j)=e_{1,0}$ for all $i\neq j$, the first identity in \eqref{eq:tff2-moments} gives
\begin{equation*}
0=\sum_{i,j=1}^N P_{(2)}(V_i,V_j)
= N+\sum_{i\neq j} P_{(2)}(V_i,V_j)
= N+(N^2-N)\,\frac{4-d\,e_{1,0}}{2(2-d)}.
\end{equation*}
Solving for $e_{1,0}$ yields
\begin{equation}\label{eq:e10_jp_en}
e_{1,0}=\frac{2(2N-d)}{d(N-1)}.
\end{equation}

Next, by \eqref{eq:P_4} (with $k=2$), we rewrite $P_{(4)}$ in the variables $(e_1,e_2)$:
\begin{equation*}
P_{(4)}'(e_1,e_2):=
1-\frac{d+2}{2} e_1+\frac{3(d+2)(d+4)}{64} e_1^2 -\frac{(d+2)(d+4)}{16} e_2,
\qquad
P_{(4)}(e_1,e_2):=\frac{8}{d(d-2)}\,P_{(4)}'(e_1,e_2).
\end{equation*}
Since $P_{(4)}$ has degree one in $e_2$, we have
\begin{equation*}
\sum_{i\neq j}P_{(4)}\bigl(e_{1,0},e_2(V_i,V_j)\bigr)
=(N^2-N)\,P_{(4)}\bigl(e_{1,0},\overline{e_2(N)}\bigr).
\end{equation*}
Therefore, the second identity in \eqref{eq:tff2-moments} becomes
\begin{equation*}
0=\sum_{i,j=1}^N P_{(4)}(V_i,V_j)
=N+(N^2-N)\,P_{(4)}\bigl(e_{1,0},\overline{e_2(N)}\bigr),
\end{equation*}
and solving for $\overline{e_2(N)}$ yields
\begin{equation}\label{eq:e2bar_jp_en}
\overline{e_2(N)}
=
\frac{d^2(d+2) +2d(d^2-4d-4)N-4(d-6)N^2}{d^2(d+2)(N-1)^2}.
\end{equation}

For each $i\neq j$, we have $y_1(V_i,V_j),y_2(V_i,V_j)\ge0$ and
$y_1(V_i,V_j)+y_2(V_i,V_j)=e_{1,0}$; hence by the AM--GM inequality,
\begin{equation}\label{eq:AMGM_pointwise}
0\le e_2(V_i,V_j)=y_1(V_i,V_j)y_2(V_i,V_j)\le \left(\frac{e_{1,0}}{2}\right)^2.
\end{equation}
Averaging \eqref{eq:AMGM_pointwise} over $i\neq j$ gives
\begin{equation*}
0\le \overline{e_2(N)}\le \frac{e_{1,0}^2}{4}.
\end{equation*}
Substituting \eqref{eq:e10_jp_en} and \eqref{eq:e2bar_jp_en} and simplifying, we obtain
\begin{equation*}
\frac{e_{1,0}^2}{4}-\overline{e_2(N)}
=
\frac{-2N(d-2)(d^2-4N)}{d^2(d+2)(N-1)^2}.
\end{equation*}
The denominator is positive and the left-hand side is $\ge 0$, hence $d^2-4N\le 0$, i.e.
\[
N\ge \frac{d^2}{4}.
\]
If $N=\frac{d^2}{4}$, then $\overline{e_2(N)}=\frac{e_{1,0}^2}{4}$.
Since each term satisfies $e_2(V_i,V_j)\le \frac{e_{1,0}^2}{4}$, attaining the average maximum forces
\[
e_2(V_i,V_j)=\frac{e_{1,0}^2}{4}\qquad(i\neq j).
\]
Thus $y_1(V_i,V_j)=y_2(V_i,V_j)=\frac{e_{1,0}}{2}$ for all $i\neq j$, so $X$ is equi-isoclinic, i.e.\ an $\EITFF_2$.
Conversely, if $X$ is an $\EITFF_2$, then the above computation implies $N=\frac{d^2}{4}$.

\textbf{(ii) Upper bound.}
The upper bound $N\le D_0+1=\binom{d+1}{2}$ follows from Corollary~\ref{cor:simplex}. Assume further that $N=\binom{d+1}{2}$.
In this case, to show that $X$ is a Grassmann $4$-design it suffices to verify
\[
\sum_{i,j=1}^N P_{(2,2)}(V_i,V_j)=0,
\]
since the moment identities~\eqref{eq:tff2-moments} already hold.
Similarly, by \eqref{eq:P_22} (with $k=2$) we rewrite $P_{(2,2)}$ in the variables $(e_1,e_2)$ and normalize it by
$P_{(2,2)}(2,1)=1$:
\[
P_{(2,2)}'(e_1,e_2):= 1-\frac{d-1}{2} e_1+\frac{(d-2)(d-1)}{2} e_2,
\qquad
P_{(2,2)}(e_1,e_2):=\frac{2}{(d-2)(d-3)}\,P_{(2,2)}'(e_1,e_2).
\]
Since $P_{(2,2)}$ has degree one in $e_2$, we obtain
\[
\sum_{i,j=1}^N P_{(2,2)}(V_i,V_j)
= N+(N^2-N)\,P_{(2,2)}\bigl(e_{1,0},\overline{e_2(N)}\bigr).
\]
Substituting \eqref{eq:e10_jp_en} and \eqref{eq:e2bar_jp_en} and simplifying, we obtain
\[
\sum_{i,j=1}^N P_{(2,2)}(V_i,V_j)
=
\frac{3N^2(d-2)^2\bigl(d(d+1)-2N\bigr)}{d^2(N-1)(d-3)(d+2)}.
\]
Hence, when $N=\frac{d(d+1)}{2}$ this sum is $0$, and $X$ is a Grassmann $4$-design. Moreover, $|X|=D_0+1=\binom{d+1}{2}$,
and since $X$ is a Grassmann $4$-design, it follows from Definition~\ref{def:tight4} that $X$ is a tight Grassmann $4$-design.

Conversely, if $X$ is a tight $4$-design, then $|X|=D_0+1=\binom{d+1}{2}$ and
Corollary~\ref{cor:GtoS} implies that $X$ is equi-chordal.
Also, a Grassmann $4$-design in particular satisfies \eqref{eq:tff2-moments}, so $X$ is a tight $2$-fusion frame.
Therefore $X$ is an $\ECTFF_2$.
\end{proof}
As a consequence, we obtain the non-existence of $\EITFF_2$ on $G_{2,d}$ in odd dimensions.

\begin{corollary}\label{cor:no-eitff-odd}
Let $d$ be odd and let $t>1$.
Then there exists no $\EITFF_t$ on $G_{2,d}$.
\end{corollary}

\begin{remark}\label{rem:zauner}
Let $d=2n$ and identify $\R^{2n}$ with $\C^n$.
Each complex line $\C z\subset\C^n$ determines a real plane
\[
V(z):=\Span_\R\{z,iz\}\in G_{2,2n}.
\]
Here a SIC-POVM in $\C^n$ means a set of unit vectors $z_1,\dots,z_{n^2}\in\C^n$ such that
$|z_i^{*}z_j|^2=1/(n+1)$ for all $i\ne j$.
Under this correspondence, a SIC-POVM yields an $\EITFF_2$ on $G_{2,2n}$ of size $n^2$, and conversely.
See~\cite{Misawa2026} for details.
\end{remark}

\subsection*{Acknowledgements}

The author would like to express his sincere gratitude to Professor Akihiro Munemasa for his guidance and encouragement.
He also thanks Professor Takayuki Okuda and Ayodeji Lindblad for their valuable comments and helpful discussions, particularly regarding constructions of designs through projective and Hopf maps.
{

}

\end{document}